\documentclass[a4paper, 12pt]{amsart}
\usepackage{amsmath, amssymb, eucal, amscd, amstext}

\usepackage{enumerate}

\newtheorem{theorem}{Theorem}[section]
\newtheorem{corollary}[theorem]{Corollary}

\newtheorem{proposition}[theorem]{Proposition}

\theoremstyle{definition}
\newtheorem{definition}[theorem]{Definition}

\theoremstyle{remark}

\newtheorem{example}[theorem]{Example}
\numberwithin{equation}{section}

\begin{document}

\baselineskip=15pt

\title[Isometric embeddings of Banach bundles]{Isometric embeddings of Banach bundles}

\author{Ming-Hsiu Hsu \and Ngai-Ching Wong}

\address{Department of Applied Mathematics, National Sun Yat-sen University,
Kaohsiung, 80424, Taiwan.} \email[Ming-Hsiu Hsu]{hsumh@math.nsysu.edu.tw}
\email[Ngai-Ching Wong]{wong@math.nsysu.edu.tw}

\thanks{This work
is jointly supported by a
Taiwan NSC Grant (NSC96-2115-M-110-004-MY3).}




\begin{abstract}
We show in this paper that every bijective linear isometry between the continuous section spaces of two non-square
 Banach bundles gives rise to a Banach bundle isomorphism.
This is to support our expectation that the geometric structure of the continuous section space of
 a Banach bundle determines completely its bundle structures.
We also describe the structure of
an \emph{into} isometry from a continuous section space into an other.
However, we demonstrate by
 an example that a non-surjective linear isometry can be far away from a subbundle embedding.
\end{abstract}

\maketitle

\section{Introduction}\label{intro}

Let $\langle B_X, \pi_X\rangle$ be a Banach bundle over a locally compact Hausdorff space $X$.
For each $x$ in $X$, denote by $B_x=\pi_X^{-1}(x)$ the fiber Banach space.
A \emph{continuous section} $f$ of the Banach bundle $\langle B_X, \pi_X\rangle$ is a continuous function
 from $X$ into $B_X$ such that $\pi_X(f(x)) = x$, i.e., $f(x)\in B_x$ for all $x$ in $X$.
Denote by $\Gamma_X$  the Banach space of all continuous sections of $\langle B_X, \pi_X\rangle$
 vanishing at infinity, i.e. those $f$ with $\lim\limits_{x\rightarrow\infty}\|f(x)\|=0$.

Let $\langle B_Y, \pi_Y\rangle$ be an other Banach bundle over a locally compact Hausdorff space $Y$ with
 continuous section space $\Gamma_Y$.
Assume that  $\Gamma_X$ is isometrically isomorphic to $\Gamma_Y$ as Banach spaces.
We want to assert whether $\langle B_X, \pi_X\rangle$ is isometrically isomorphic to $\langle B_Y, \pi_Y\rangle$ as
 Banach bundles (see \S 2 for definitions).
In other words, we expect that the geometric structure of the continuous sections of a Banach bundle determines
 its bundle structure.

\begin{example}\label{ex:TriLine}
(Trivial line bundles)
Let $B_X = X\times \mathbb K$ and $B_Y=Y\times \mathbb K$, where the underlying field $\mathbb K$ is either
 the real $\mathbb R$ or the complex $\mathbb C$.
The continuous section spaces are $C_0(X)$ and $C_0(Y)$, the Banach spaces of continuous scalar functions vanishing
 at infinity, respectively.
The classical Banach-Stone Theorem (see, e.g., \cite{Beh79}) asserts that every linear isometry $T$ from
 $C_0(X)$ onto $C_0(Y)$ is a weighted composition operator:
\begin{eqnarray}\label{eq:wco}
Tf(y)=h(y)f(\varphi(y)), \quad\forall \ f\in C_0(X), \ y\in Y.
\end{eqnarray}
Here, $\varphi$ is a homeomorphism from $Y$ onto $X$, and $h$ is a continuous scalar function on $Y$ with
 $|h(y)|=1$, $\forall \,y\in Y$.
This induces an isometric bundle isomorphism $\Phi:B_X\rightarrow B_Y$ from $B_X=X\times\mathbb K$ onto
 $B_Y=Y\times\mathbb K$ defined by
 $$
 \Phi(x,\alpha)=(\varphi^{-1}(x),h(\varphi^{-1}(x))\alpha), \quad \forall \ (x,\alpha)\in X\times\mathbb K.
 $$
Hence, the trivial line bundles $\langle X\times\mathbb K,\pi_X\rangle$ and $\langle Y\times\mathbb K,\pi_Y\rangle$
 are isometrically isomorphic if and only if they have isometrically isomorphic continuous section spaces.
\end{example}

Recall that a Banach space $E$ is \emph{strictly convex} if $\|x+y\|<2$ whenever
$x\neq y$ in $E$ with $\|x\|=\|y\|=1$.
A Banach space $E$ is said to be \emph{non-square} if $E$ does not contain a copy of the two-dimensional space
 ${\mathbb K}\oplus_\infty{\mathbb K}$ equipped with the norm $\|(a,b)\|=\max\{|a|,|b|\}$.
In other words, if $x$ and $y$ are unit vectors in $E$, at least one of $\|x+y\|$ and $\|x-y\|$ is less
 than $2$.
Note that a Banach space $E$ is non-square if $E$ or its dual $E^*$ is strictly convex.

\begin{example}\label{ex:Tri}
(Trivial bundles)
Let $E$ and $F$ be Banach spaces.
We consider the trivial bundles $B_X= X\times E$ and $B_Y=Y\times F$.
The continuous section spaces are $C_0(X,E)$ and $C_0(Y,F)$, the Banach spaces of continuous vector-valued functions
 vanishing at infinity, respectively.
If $E$ and $F$ are strictly convex, by a result of Jerison \cite{Jerison} we know that every linear isometry $T$ from
 $C_0(X,E)$ onto $C_0(Y,F)$ is of the form:
\begin{eqnarray}\label{eq:wco2}
Tf(y)=h_y(f(\varphi(y))), \quad\forall \ f\in C_0(X,E), \ y\in Y.
\end{eqnarray}
Here, $\varphi$ is a homeomorphism from $Y$ onto $X$, and $h_y$ is a linear isometry
 from $E$ onto $F$ for all $y$ in $Y$.
Moreover, the map $y\mapsto h_y$ is SOT continuous on $Y$. In the
case both the Banach dual spaces $E^*$ and $F^*$ are strictly
convex, Lau gets the same representation
 \eqref{eq:wco2} in \cite{Lau}.
It is further extended that the same conclusion holds whenever $E$ and $F$ are non-square in \cite{Jeang03}
 or the centralizers of $E$ and $F$ are one dimensional in \cite{Beh79}.
The representation (1.2) induces an isometric bundle isomorphism $\Phi:B_X\rightarrow B_Y$ from $B_X=X\times E$ onto
 $B_Y=Y\times F$ defined by
 $$
 \Phi(x,e)=(\varphi^{-1}(x),h_{\varphi^{-1}(x)}(e)), \quad \forall \ (x,e)\in X\times E.
 $$
Hence, the trivial bundles $\langle X\times E,\pi_X\rangle$ and $\langle Y\times F,\pi_Y\rangle$ are
 isometrically isomorphic if and only if they have isometrically isomorphic continuous section spaces.
We note that if $E$ or $F$ is not non-square, the above assertion \eqref{eq:wco2} can be false
 as shown in Example \ref{ex:onto}.
\end{example}

In this paper, we discuss the general Banach bundle case.
Motivated by Example \ref{ex:Tri}, we call a Banach bundle $\langle B_X, \pi_X\rangle$ \emph{non-square}
 (resp.\ \emph{strictly convex}) if every fiber Banach space $B_x=\pi_X^{-1}(x)$ is non-square
 (resp.\ strictly convex).
The proof of the following theorem will be given in Section \ref{pf}.

\begin{theorem}\label{thm:onto}
Two non-square Banach bundles $\langle B_X, \pi_X\rangle$ and $\langle B_Y, \pi_Y\rangle$ are isometrically isomorphic
 as Banach bundles if and only if their continuous section spaces $\Gamma_X$ and $\Gamma_Y$ are isometrically
 isomorphic as Banach spaces.
\end{theorem}

We also consider the case when the continuous section space $\Gamma_X$ is embedded into $\Gamma_Y$
 as a Banach subspace.
We want to see whether $\langle B_X,\pi_X\rangle$ embedded into $\langle B_X,\pi_X\rangle$ as a subbundle.
Assume $F$ is strictly convex.
It is shown in \cite{Cam78b, Hol66, Jeang97} that every linear isometry from $C_0(X,E)$ \emph{into} $C_0(Y,F)$
 induces a continuous function $\varphi$ from a nonempty subset $Y_1$ of $Y$ onto $X$
 and a field $y\mapsto h_y$ of norm one linear operators from $E$ into $F$ on $Y_1$, such that
\begin{eqnarray}\label{eq:into1}
Tf(y)=h_y(f(\varphi(y))), \quad\forall \ f\in C_0(X,E), \ y\in Y_1,
\end{eqnarray}
and
\begin{eqnarray}\label{eq:into2}
\|Tf|_{Y_1}\|=\sup\{\|Tf(y)\|: y\in Y_1\} = \|f\|, \quad \forall \ f\in C_0(X,E).
\end{eqnarray}
When $F$ is not strictly convex, the conclusion does not hold (see \cite{Jeang97}).

In Theorem \ref{thm:into}, we extend the above representation \eqref{eq:into1} and \eqref{eq:into2} to the
 general strictly convex Banach bundle case.
Supposing all $h_y$ are isometries, we can consider $\langle
B_X,\pi_X\rangle$ to be embedded into
 $\langle B_Y,\pi_Y\rangle$ as a subbundle.
However, in Example \ref{ex:into} we have a linear into isometry between trivial bundles with all fiber maps
 $h_y$ not being isometric.

\section{Preliminaries}

Let $\mathbb K=\mathbb R$ or $\mathbb C$ be the underlying field.
Let $X$ be a locally compact Hausdorff space.
A \emph{Banach bundle} (see, e.g. \cite{Fell}) over $X$ is a pair $\langle B_X,\pi_X\rangle$,
 where $B_X$ is a topological space and $\pi_X$ is a continuous open surjective map from $B_X$ onto $X$, such that,
 for all $x$ in $X$, each fiber $B_x=\pi_X^{-1}(x)$ carries a Banach space structure
 in the subspace topology and satisfying
 the following conditions:
\begin{enumerate}[(1)]
    \item Scalar multiplication, addition and the norm on $B_X$ are all continuous wherever they are defined.
    \item If $x\in X$ and $\{b_i\}$ is any net in $B_X$ such that $\|b_i\|\rightarrow0$ and
          $\pi(b_i)\rightarrow x$ in $X$, then $b_i\rightarrow0_x$ (the zero element of $B_x$) in $B_X$.
\end{enumerate}
The condition (2) above ensures that the zero section is in $\Gamma_X$.

\begin{definition}(\cite{Fell}, p.128)\label{def:iso}
A Banach bundle $\langle B_X,\pi_X\rangle$ is said to be \emph{isometrically isomorphic} to a Banach bundle
 $\langle B_Y,\pi_Y\rangle$ if there are homeomorphisms $\Phi : B_X \to B_Y$ and $\psi : X \to Y$ such that
\begin{enumerate}[(a)]
    \item $\pi_Y\circ\Phi=\psi\circ\pi_X$, i.e., $\Phi(B_x)=B_{\psi(x)}, \ \forall \ x\in X$;
    \item $\Phi|_{B_x}$ is a linear map from $B_x$ onto $B_{\psi(x)}, \ \forall \ x\in X$;
    \item $\Phi$ preserves norm, i.e., $\|\Phi(b)\| = \|b\|, \ \forall \ b\in B_X$.
\end{enumerate}
\end{definition}

Clearly, all the fiber linear maps $\Phi|_{B_x}$ are surjective isometries.
In fact, an isometrical bundle isomorphism  $(\Phi,\psi): \langle B_X,\pi_X\rangle \to \langle B_Y,\pi_Y\rangle$
 induces a linear isometry $T$ from $\Gamma_X$ onto $\Gamma_Y$ defined
 by setting $\varphi=\psi^{-1}: Y\to X$, $h_y=\Phi|_{B_{\varphi(y)}}: B_{\varphi(y)}\to B_y$, and
\begin{eqnarray}\label{eq:basicform}
Tf(y) = \Phi(f(\varphi(y)))= h_y(f(\varphi(y))), \quad\forall \ f\in\Gamma_X,\ y\in Y.
\end{eqnarray}
In other words, isometrically isomorphic Banach bundles have isometrically isomorphic continuous section spaces.
We want to establish the converse of this observation.

In general, let $\varphi:Y\rightarrow X$ be a continuous  map, and let
 $y\mapsto h_y$ be a field of  fiber linear maps  $h_y: B_{\varphi(y)}\to B_y$, $\forall \,y\in Y$.
 We can define a linear map $T$ sending vector sections $f$ in $\langle B_X, \pi_X\rangle$ to
 vector sections $Tf$ in $\langle B_Y, \pi_Y\rangle$  by setting
 $Tf(y)=h_y(f(\varphi(y))), \forall y\in Y$.
The field $y\mapsto h_y$ is said to be \emph{continuous} if  $y_{\lambda}\rightarrow y$ implies
 $h_{y_{\lambda}}(f(\varphi(y_{\lambda})))\rightarrow h_y(f(\varphi(y)))$, and \emph{uniformly bounded}
 if $\sup_{y\in Y} \|h_y\| < +\infty$. When $B_X=X\times E$ and $B_Y=Y\times F$, the continuity of a  field  $y\mapsto h_y$ of
fiber linear maps reduces to the usual SOT continuity.
In general, assuming $\varphi$ is proper, i.e.,
 $\lim_{y\to\infty}\varphi(y)=\infty$,
 if the field $y\mapsto h_y$ is uniformly bounded and
 continuous on $Y$, then
 $T(\Gamma_X)\subseteq \Gamma_Y$.
Conversely, we will see in Theorem \ref{thm:into} that every linear into isometry
$T: \Gamma_X\to\Gamma_Y$ defines a
   continuous field  $y\mapsto h_y$ of fiber linear maps with all $\|h_y\|=1$, provided that
   $\langle B_Y, \pi_Y\rangle$ is
strictly convex.

In terms of Banach bundles, Example \ref{ex:TriLine} says that trivial line bundles are completely determined
 by the geometric structure of its continuous sections.
It is also the case for trivial Banach bundles $X\times E$ and $Y\times F$ whenever $E$ and $F$ are non-square,
 as demonstrated in Example \ref{ex:Tri}.
In attacking the general Banach bundle case, we need the following result of Fell \cite{Fell}.

\begin{proposition}$($\cite{Fell}, $p.129)$\label{prop}
Let $\{s_i\}$ $(i\in I)$ be a net of elements of $B_X$ and $s$ an element of $B_X$ such that
 $\pi_X(s_i)\rightarrow\pi_X(s)$. Suppose further that for each $\epsilon>0$ we can find a net $\{u_i\}$ of
 elements of $B_X$ (indexed by the same $I$) and an element $u$ of $B_X$ such that:
 $(1)$ $u_i\rightarrow u$ in $B_X$,
 $(2)$ $\pi_X(u_i)=\pi_X(s_i)$ for each $i$,
 $(3)$ $\|s-u\|<\epsilon$, and
 $(4)$ $\|s_i-u_i\|<\epsilon$ for all large enough $i$.
Then $s_i\rightarrow s$ in $B_X$.
\end{proposition}

\section{The Results}\label{pf}

First, we discuss the \emph{into} isometry case.  We shall write  $E^*$ and $S_E$ for the Banach dual space and the
unit sphere of a Banach space $E$, respectively.

\begin{theorem}\label{thm:into}
Suppose $\langle B_X, \pi_X\rangle$ and $\langle B_Y, \pi_Y\rangle$ are Banach bundles such that
 $\langle B_Y, \pi_Y\rangle$ is strictly convex.
Let $T:\Gamma_X\rightarrow\Gamma_Y$ be a linear \emph{into} isometry.
Then there exist a continuous map $\varphi$ from a nonempty subset $Y_1$ of $Y$ onto $X$,
 and a field of norm one linear operators $h_y:B_{\varphi(y)}\rightarrow B_y$, for all $y$ in $Y_1$, such that
$$
Tf(y)=h_y(f(\varphi(y))), \quad \forall \ f \in \Gamma_X, \ y \in Y_1,
$$
and
$$
\|Tf|_{Y_1}\|=\|Tf\|, \quad \forall \ f \in\Gamma_X.
$$
\end{theorem}

\begin{proof}
We employ the notations developed in \cite{Jeang97, Jeang03}.
For $x$ in $X$, $y$ in $Y$, $\mu$ in $S_{B_x^*}$ and $\nu$ in $S_{B_y^*}$, we set
$$
S_{x,\mu}=\{f\in\Gamma_X : \ \mu(f(x))=\|f\|=1\},
$$
$$
R_{y,\nu}=\{g\in\Gamma_Y : \ \nu(g(y))=\|g\|=1\},
$$
$$
Q_{x, \mu}=\{y\in Y:T(S_{x,\mu})\subseteq R_{y,\nu} \text{ for some $\nu$ in $S_{B_y^*}$}\},
$$
and
$$
Q_x=\bigcup\limits_{{\mu\in S_{B_x^*}}}Q_{x, \mu}.
$$
As in \cite{Jeang03}, it is not difficult to see that
\begin{enumerate}[(a)]
    \item For all $x$ in $X$, the set $S_{x,\mu}\neq\emptyset$ for some $\mu$ in $S_{B_x^*}$;
    \item If $S_{x, \mu}\neq\emptyset$, then so is $Q_{x, \mu}$.
\end{enumerate}
By the strict convexity of $\langle B_Y,\pi_Y\rangle$, we have
\begin{enumerate}[(a)]\setcounter{enumi}{2}
    \item $Q_{x_1}\bigcap Q_{x_2}=\emptyset$ for all $x_1\neq x_2$.
\end{enumerate}
Set
$$
Y_1=\bigcup\limits_{x\in X}Q_x=\bigcup\limits_{x\in X}\bigcup\limits_{{\mu\in S_{B_x^*}}}Q_{x, \mu}.
$$
   From (c), we can define a map $\varphi$ from $Y_1$ onto $X$ by
$$
\varphi(y)=x \quad \mbox{if} \quad y\in Q_x.
$$
Using the strict convexity of $\langle B_Y,\pi_Y\rangle$ again, we also have
\begin{enumerate}[(a)]\setcounter{enumi}{3}
    \item $f(\varphi(y))=0$ implies $Tf(y)=0$, i.e. $\ker\delta_{\varphi(y)}\subseteq\ker(\delta_y\circ T)$.
\end{enumerate}
Then there exists a linear operator $h_y:B_{\varphi(y)}\rightarrow B_y$ such that
$$
\delta_y\circ T=h_y\circ\delta_{\varphi(y)}, \quad \forall \ y \in Y_1.
$$
In other words,
$$
Tf(y)=h_y(f(\varphi(y))), \quad \forall \ f \in \Gamma_X, \ y \in Y_1.
$$

For all $b$ in $B_{\varphi(y)}$, choose an element $f$ in $\Gamma_X$ such that $f(\varphi(y))=b$ and $\|f\|=\|b\|$.
It follows
$$
\|h_y(b)\|=\|h_y(f(\varphi(y)))\|=\|Tf(y)\|\leq\|Tf\|=\|f\|=\|b\|.
$$
Since $y\in Y_1$, there exist $x$ in $X$, $\mu$ in $S_{B_x^*}$ and $\nu$ in $S_{B_y^*}$ such that
$$
\nu(Tf(y))=\mu(f(x))=1, \quad \forall\ f \in S_{x,\mu},
$$
and hence
$$
\|h_y(f(x))\|=\|Tf(y)\|=1.
$$
This shows that $\|h_y\|=1$.

For all $f$ in $\Gamma_X$ with norm one, $f\in S_{x, \mu}$ for some $x$ and $\mu$.
As a result, $Tf\in R_{y,\nu}$ for some $y$ in $Y_1$ and $\nu$ in $S_{B_y^*}$.
Thus,
$$
\nu(Tf(y))=\mu(f(x))=\|f\|=1.
$$
Therefore, $\|Tf|_{Y_1}\|=1=\|f\|=\|Tf\|$.

It remains to show that the map $\varphi$ is continuous.
Let $y_{\lambda}$ be a net converging to $y$ in $Y_1$.
If $\varphi(y_{\lambda})$ does not converge to $\varphi(y)$, then by passing to a subnet if necessary,
 we can assume it converges to an $x\neq\varphi(y)$ in $X_{\infty}$.
Let $U_1$ and $U_2$ be disjoint neighborhoods of $x$ and $\varphi(y)$ in $X_{\infty}$, respectively.
Let $f$ be an element of $S_{\varphi(y), \mu}$ supporting in $U_2$.
Then $f(\varphi(y_{\lambda}))=0$ for large $\lambda$.
By (d), $Tf(y_{\lambda})=0$ for large $\lambda$.
The definition of $\varphi$ implies that there exists a $\nu$ in $S_{B_y^*}$ such that
 $\nu(Tf(y))=\mu(f(\varphi(y)))=\|f\|=1$.
Hence, $\|Tf(y)\|=1$, contradicting to the fact $Tf(y_{\lambda})=0$ for large $\lambda$.
\end{proof}

\begin{example}\label{ex:into}
For each $\theta$ in $[0,2\pi]$, let $P_{\theta}:\mathbb{R}^2\rightarrow\mathbb{R}^2$ be the orthogonal projection
 onto the one-dimensional subspace of $\mathbb R^2$ spanned by the unit vector $(\cos\theta,\sin\theta)$.
Every element $f$ in $C(\{0\},\mathbb{R}^2)$ is given by the vector $f(0)=(r\cos t,r\sin t)$ for some $r\geq0$ and
 $t\in[0,2\pi]$.
Define a linear isometry $T:C(\{0\},\mathbb{R}^2)\rightarrow C([0,2\pi],\mathbb{R}^2)$ by
\begin{align*}
T(f)(\theta) &= P_{\theta}(f(0))=P_{\theta}(r\cos t,r\sin t)\\
             &= (r\cos(t-\theta)\cos\theta,r\cos(t-\theta)\sin\theta).
\end{align*}
In the notations of Theorem \ref{thm:into}, $h_\theta = P_\theta$, $Y_1 = Y = [0,2\pi]$, and
$$
Tf(\theta)=h_\theta(f(0)), \quad \forall \ f \in C(\{0\},\mathbb{R}^2), \ \theta \in [0,2\pi].
$$
Note that $h_{\theta}=P_{\theta}$  is not an isometry for every
$\theta$
 in $Y_1=[0,2\pi]$.
\end{example}

Here comes the proof of our main result.

\begin{proof}[Proof of Theorem \ref{thm:onto}]
Let $\langle B_X,\pi_X\rangle$ and $\langle B_Y,\pi_Y\rangle$ be two non-square Banach bundles and
 $T$ a linear isometry from $\Gamma_X$ onto $\Gamma_Y$.
Denote by
$$
K_X=\bigcup\limits_{x\in X}(\{x\}\times U_{B_x^*}) \quad \mbox{and} \quad
K_Y=\bigcup\limits_{y\in Y}(\{y\}\times U_{B_y^*}),
$$
 the disjoint unions of the compact sets $\{x\}\times U_{B_x^*}$ and $\{y\}\times U_{B_y^*}$, respectively.
Note that both the Hausdorff spaces $K_X$ and $K_Y$ are locally compact.
Define a linear isometry $\Psi:\Gamma_Y\rightarrow C_0(K_Y)$ by
$$
 \Psi(g)(y, \nu)= \nu(g(y)), \quad\forall\, g\in \Gamma_Y,  (y,\nu)\in K_Y.
$$
Then $\widetilde{T}=\Psi\circ T$ is a linear isometry from $\Gamma_X$ into $C_0(K_Y)$.
By Theorem \ref{thm:into}, there exist a continuous map $\widetilde{\varphi}$ from a nonempty subset $A_Y$ of
 $K_Y$ onto $X$ and bounded linear functionals $\widetilde{h}_{(y,\nu)}\in B_{\widetilde{\varphi}(y,\nu)}^*$ such that
\begin{align}\label{eq:3.1}
\widetilde{T}f(y,\nu)=\nu(Tf(y))=\widetilde{h}_{(y,\nu)}(f(\widetilde{\varphi}(y,\nu))),
\quad \forall\ f\in\Gamma_X, \ (y,\nu)\in A_Y.
\end{align}
Applying the same argument to $T^{-1}$, there exist a continuous map $\widetilde{\psi}$ from a subset $A_X$ of
 $K_X$ onto $Y$ and bounded linear functionals $\widetilde{k}_{(x,\mu)}$ in $B_{\widetilde{\psi}(x,\mu)}^*$ such that
$$
\mu(T^{-1}g(x))=\widetilde{k}_{(x,\mu)}(g(\widetilde{\psi}(x,\mu))),
\quad \forall\  g\in\Gamma_Y, \ (x,\mu)\in A_X.
$$

Let
$$
C_y=\{\nu\in S_{B_y^*}:(y,\nu)\in A_Y\},
$$
$$
X_I=\{x\in X:\ \mbox{there exists a}\ \mu \ \mbox{in} \ S_{B_x^*}\ \mbox{such that}\ (x,\mu)\in A_X\},
$$
and
$$
Y_I=\{y\in Y:\ \mbox{there exists a}\ \nu \ \mbox{in} \ S_{B_y^*}\ \mbox{such that}\ (y,\nu)\in A_Y\}.
$$
We make the following easy observations:
\begin{enumerate}[(I)]
    \item $X_I=X$ and $Y_I=Y$;
    \item $C_y$ is total in $B_y^*$, for all $y$ in $Y$.
\end{enumerate}

By modifying the arguments in \cite{Jeang03}, it is not difficult to show that if $\langle B_Y,\pi_Y\rangle$
 is non-square, $\widetilde{\varphi}(y,\nu_1)=\widetilde{\varphi}(y,\nu_2)$ for all $\nu_i$ in $C_y$ and
 for all $y$ in $Y$.
Consequently, we can define a continuous map $\varphi:Y\rightarrow X$ by
$$
\varphi(y)=\widetilde{\varphi}(y,\nu), \ \mbox{for some}\ \nu\in C_y.
$$
In view of \eqref{eq:3.1} and (II), we have $f(\varphi(y))=0$ implies $Tf(y)=0$.
Then there exists a bounded linear operator $h_y:B_{\varphi(y)}\rightarrow B_y$ such that
\begin{eqnarray}\label{eq:main}
Tf(y)=h_y(f(\varphi(y))), \ \forall \ f \in \Gamma_X, \ y \in Y.
\end{eqnarray}

By symmetry, $T^{-1}$ also carries a form
$$
T^{-1}g(x)=k_x(g(\psi(x))), \ \forall \ g\in\Gamma_Y, \ x\in X,
$$
 for some continuous map $\psi$ from $X$ onto $Y$, and bounded linear operators $k_x$ from $B_{\psi(x)}$ into $B_x$,
 for all $x$ in $X$.
Consequently,
$$
f(x)=(T^{-1}(Tf))(x)=k_x(Tf(\psi(x)))=k_xh_{\psi(x)}(f(\varphi(\psi(x)))).
$$
This implies that $\varphi$ is a homeomorphism with inverse $\psi$,
 and $h_y$ are bijective linear isometries with
 inverses $k_{\varphi(y)}$  for all $y$ in $Y$.

Let $\Phi=(h_y^{-1})_{y\in Y}$, i.e. $\Phi|_{B_y}=h_y^{-1}$.
Then it defines a map from $B_Y$ onto $B_X$ as follows:
 for all $b$ in $B_Y$ and $\pi_Y(b)=y_0$.
Choose a continuous section $g$ in $\Gamma_Y$ such that $g(y_0)=b$.
Then
$$
\Phi(b)=h_{y_0}^{-1}(g(y_0)) = T^{-1}(g)(\varphi(y_0)).
$$
We show that $\Phi$ is a homeomorphism from $B_Y$ to $B_X$.
By symmetry, it suffices to prove that $\Phi$ is continuous.
We shall make use of Proposition \ref{prop} in below.

Let
$b_i\rightarrow b$ in $B_Y$. We show that
$\Phi(b_i)\rightarrow\Phi(b)$ in $B_X$. Since $\pi_Y$ and $\varphi$
are continuous, we have $\pi_Y(b_i)\rightarrow\pi_Y(b)$ and
 $\varphi(\pi_Y(b_i))\rightarrow\varphi(\pi_Y(b))$.
Let $s_i=\Phi(b_i)$ and $s=\Phi(b)$. Choose a continuous section $g$
in $\Gamma_Y$ such that $g(\pi_Y(b))=b$. Then, for all $\epsilon>0$,
we have $\|g(\pi_Y(b_i))-b_i\|<\epsilon$ for all large enough $i$.
The fact $\pi_X\circ\Phi=\varphi\circ\pi_Y$ (this follows from
\eqref{eq:main}) implies that
 $\pi_X(s_i)=\pi_X(\Phi(b_i))=\varphi(\pi_Y(b_i))$ approaches $\varphi(\pi_Y(b))=\pi_X(\Phi(b))=\pi_X(s)$.
Let $u_i=\Phi(g(\pi_Y(b_i)))$ and $u=\Phi(g(\pi_Y(b)))$. Since
$\Phi|_{B_y}$ is an isometry, we have
$\|u_i-s_i\|=\|g(\pi_Y(b_i))-b_i\|<\epsilon$, for all large enough
$i$. And
$$
u_i=\Phi(g(\pi_Y(b_i)))=h_{\pi_Y(b_i)}^{-1}(g(\pi_Y(b_i)))=f(\varphi(\pi_Y(b_i))),
$$
 for some $f$ in $\Gamma_X$,
 which converges to
$$
f(\varphi(\pi_Y(b)))=h_{\pi_Y(b)}^{-1}(g(\pi_Y(b)))=\Phi(b)=u
$$
 in $B_X$.
By Proposition \ref{prop}, we have $\Phi(b_i)=s_i\rightarrow
s=\Phi(b)$ in $B_X$. This shows that $\Phi$ is continuous and
complete the proof of Theorem \ref{thm:onto}.
\end{proof}

\begin{corollary}
Assume $\langle B_X, \pi_X\rangle$ and $\langle B_Y, \pi_Y\rangle$ are two non-square Banach bundles over
locally compact Hausdorff spaces with isometrically isomorphic continuous sections.
If $\langle B_X,\pi_X\rangle$ is locally trivial, then so is
 $\langle B_Y,\pi_Y\rangle$.
\end{corollary}

The following example shows that the conclusion in Theorem \ref{thm:onto} might not hold if
 $\langle B_X,\pi_X\rangle$ or $\langle B_Y,\pi_Y\rangle$ is not non-square.

\begin{example}\label{ex:onto}
Let $\pi_i$ be the $i$-th coordinate map of $\mathbb R\oplus_{\infty}\mathbb R$, $i=1,2$.
Each element $f$ in $C(\{0\},\mathbb R\oplus_{\infty}\mathbb R)$ is given by the vector $f(0)$ in
 $\mathbb R\oplus_{\infty}\mathbb R$.
Define a linear map $T:C(\{0\},\mathbb R\oplus_{\infty}\mathbb R)\rightarrow C(\{1,2\},\mathbb R)$ by
$$
Tf(i)=\pi_i(f(0)), \quad \forall \ f\in C(\{0\},\mathbb R\oplus_{\infty}\mathbb R),\ i=1,2.
$$
It is easy to see that $T$ is an isometrical isomorphism, but $\mathbb R\oplus_{\infty}\mathbb R$ and $\mathbb R$
 are not isomorphic as Banach spaces.
In particular, $\mathbb R\oplus_{\infty}\mathbb R$ is not isometrically isomorphic to $\{1,2\}\times\mathbb R$
 as Banach bundles, although they have isometrically isomorphic continuous section spaces.
\end{example}

\end{document}